\documentclass[preprint,12pt]{elsarticle}

\usepackage{amssymb}
\usepackage{url}
\usepackage{amsthm}
\usepackage{amsmath}

\newtheorem{dfn}{Definition}
\newtheorem{lem}{Lemma}
\newtheorem{obs}{Observation}
\newtheorem{thm}{Theorem}

\newtheorem{pro}{Proposition}
\newtheorem{cor}{Corollary}

\newtheorem{hyp}{Hypothesis}

\journal{}

\begin{document}

\begin{frontmatter}

\title{Edge and mixed metric dimension of Johnson graphs}

\author{Jozef Kratica \fnref{mi}}
\ead{jkratica@mi.sanu.ac.rs}
\author{Mirjana \v{C}angalovi\'c \fnref{fon}}
\ead{mirjana.cangalovic@alumni.fon.bg.ac.rs}
\author{Vera Kova\v{c}evi\'c-Vuj\v{c}i\'c \fnref{fon}}
\ead{vera.vujcic@alumni.fon.bg.ac.rs}
\author{Milica Milivojevi\'c Danas\fnref{kg}}
\ead{milica.milivojevic@kg.ac.rs}

 \address[mi]{Mathematical Institute, Serbian Academy of Sciences and Arts, Kneza Mihaila 36/III, 11 000 Belgrade, Serbia}

 \address[fon]{Faculty of Organizational Sciences, University of  Belgrade, Jove Ili\'ca 154, 11000 Belgrade, Serbia}

 \address[kg]{Faculty of Engineering, University of Kragujevac, Sestre Janji\' c 6, Kragujevac, Serbia }

\begin{abstract}
In this paper, both edge and mixed metric dimensions of Johnson graphs $J_{n,k}$ are considered.
A new tight lower bound for $\beta_E(J_{n,k})$ based on hitting sets has been obtained.
Using this bound, exact values for $\beta_E(J_{n,2})$ and $\beta_M(J_{n,2})$ have
been derived, and it is proved that $\beta_E(J_{n,2}) = \beta_M(J_{n,2})$.
\end{abstract}


\end{frontmatter}


\section{Introduction}

Resolving sets and the metric dimension were introduced to the graph theory by Slater (1975) in \cite{metd1}
and independently by Harary and Melter (1976) in \cite{metd2}. For a simple connected undirected graph $G = (V,E)$ the distance between vertices $u$ and $v$ is denoted by $d(u,v)$. Vertex $w$ of graph $G$ is said to resolve two vertices
$u,v \in V(G)$ if $d(u,w) \ne d(v,w)$. Set $S \subseteq V(G)$ is called a resolving set of $G$ if
any pair of distinct vertices of $G$ is resolved by some vertex from $S$. The metric dimension of $G$,
denoted by $\beta(G)$, is the minimal cardinality of its resolving sets, while the metric basis is
any resolving set with cardinality $\beta(G)$. The metric dimension has applications in several diverse areas:
the robot navigation \cite{metd3}, network discovery and verification \cite{ber06}, geographical routing protocols \cite{liu06},
computer graphic \cite{appl4} and chemistry \cite{appl5,appl6}, etc.

Since the problem of computing the metric dimension of a graph is NP-hard in a general case (see \cite{metd3}),
many papers in the literature have been devoted to finding its exact value or good bounds for some
classes of graphs. In particular, the metric dimension and resolving sets of Johnson and Kneser graphs
have been studied in \cite{metdj1,metdj}. Moreover,
in \cite{metdj} some interesting combinatorial objects were used to obtain resolving sets of these graphs.
For Johnson graphs these objects are projective planes and symmetric designs,
while for Kneser graphs partial geometries, Hadamard matrices, Steiner
systems and toroidal grids were used.

However, recently, several authors have turned their attention in the opposite direction from resolvability.
One such graph invariant is equidistant dimension, introduced by Gonz{\'a}lez et al. (2022) in \cite{eqdim1},
which is also NP-hard in a general case (see \cite{eqdimnp}). Set $S \subseteq V(G)$ is called a
distance-equalizer set of $G$ if any pair of distinct vertices $x$ and $y$ from $V(G) \setminus S$
there exists a vertex $w \in S$ at same distance from $x$ and $y$.

The edge metric dimension invariant was recently introduced by Kelenc et al. (2018) in \cite{emetd1} in order to extend
the concept of resolving to edges. In the sequel, edge $e \in E(G)$ will be denoted by $e = uv$,
where $u,v \in V(G)$ are its endpoints. The distance between edge $e=uv$ and vertex $w$ is defined as
$d(e,w) = min \{d(u,w), d(v,w)\}$. Two edges $e_1$ and $e_2$ are resolved by vertex $w$ if
$d(e_1,w) \ne d(e_2,w)$. Set $N \subseteq V(G)$ is called an edge resolving set of $G$ if
any pair of distinct edges of $G$ is resolved by some vertex from $N$. The edge metric dimension of $G$,
denoted by $\beta_E(G)$, is the minimal cardinality of its edge resolving sets, while the edge metric basis is
any edge resolving set with cardinality $\beta_E(G)$.

The next step was unification of vertices and edges within an extended resolving concept,
which was made by Kelenc et al. (2017) in \cite{mmetd1}.
They introduced the concept of the mixed metric dimension and studied its theoretical and practical properties.
Two items $a, b \in V(G) \cup E(G)$ are resolved by vertex $w$ if
$d(a,w) \ne d(b,w)$. Set $M \subseteq V(G)$ is called a mixed resolving set of $G$ if
any pair of distinct items from $V(G) \cup E(G)$ is resolved by some vertex from $M$. The mixed metric dimension of $G$,
denoted by $\beta_M(G)$, is the minimal cardinality of its mixed resolving sets, while the mixed metric basis is
any mixed resolving set with cardinality $\beta_M(G)$. In general case finding $\beta_E(G)$ and $\beta_M(G)$ is NP-hard
\cite{emetd1,mmetd1}, respectively.

An efficient variable neighbourhood search (VNS) metaheuristic approach to metric dimension problem is
presented in \cite{vns}. The same approach with slight modifications can be used for
obtaining upper bound of edge and mixed metric dimension problems.

Next, we will present several lower bounds of edge and mixed metric dimension from the literature.

Let:
\begin{itemize}
\item $deg_v$ is degree of vertex $v$;
\item $\delta(G)$ is minimal vertex degree in graph, i.e. $\delta(G) = min \{deg_v | v \in V(G)\}$;
\item $\Delta(G)$ is maximal vertex degree in graph, i.e. $\Delta(G) = max \{deg_v | v \in V(G)\}$;
\item $Diam(G)$ is diameter of graph, i.e. $Diam(G) = max \{d(u,v) | u,v \in V(G) \}$.
\end{itemize}

\begin{pro}\mbox{\rm(\cite{mmetd1,mmetd2})} \label{elb1} Let $G$ be a connected graph. Then \begin{equation}\label{Y04}\beta_E(G)\geq \lceil\log_2\triangle (G)\rceil.\end{equation}
\end{pro}

\begin{thm}\label{elb2} \mbox{\rm(\cite{fil19})} Let $G$ be a connected graph, then \begin{equation}\label{Y02}\beta_E(G) \geq 1 + \lceil \log_2 \delta(G) \rceil. \end{equation}\end{thm}

\begin{pro} \mbox{\rm(\cite{mmetd1,mmetd2})} \label{mlb1} For any graph $G$ it holds
	\begin{equation}\label{Y01}\beta_{M}(G)\geq \max\{\beta(G),\beta_{E}(G)\}. \end{equation}
\end{pro}

\begin{thm} \label{mlb2} \mbox{\rm(\cite{mil21,mil21a})} Let $G$ be a connected graph and let $x$ be an arbitrary vertex from mixed resolving set $S$ of  $G$. Then,
$|S| \geq 1 + \lceil \log_2 (1+\deg_x) \rceil$.\end{thm}

\begin{cor} \label{mlb3} \mbox{\rm(\cite{mil21,mil21a})} Let $G$ be an $r$-regular graph, then \begin{equation}\beta_M(G) \geq 1 + \lceil \log_2 (r+1) \rceil.\end{equation}\end{cor}

\begin{thm}\label{mlb4} \mbox{\rm(\cite{mil21,mil21a})} Let $G=(V,E)$ be a connected graph with mixed metric dimension $\beta_M (G)$, then \begin{equation}\label{Y03}|V|+|E|\leq Diam(G)^{\beta_M (G)}+\beta_M (G)(\bigtriangleup (G)+1). \end{equation}
\end{thm}

\begin{dfn} \label{wset} \mbox{\rm(\cite{bal09})} For an arbitrary edge $e = uv \in E(G)$
$$W_{uv}=\{w \in V(G)| d(u,w)<d(v,w)\}$$
and
$$W_{vu}=\{w \in V(G)| d(v,w)<d(u,w)\}$$
\end{dfn}

\begin{lem} \label{mlb5} \mbox{\rm(\cite{mil21,mil21a})}
	Let $G$ be a connected graph, $uv\in E$ an arbitrary edge and $S$ a mixed resolving set, then
 \begin{itemize}
  \item [a)] $W_{vu}\cap S\neq 0;$
   \item [b)] $W_{uv}\cap S\neq 0$.
  \end{itemize}
\end{lem}

In the sequel we shall use the well-known concept of hitting sets \cite{hit}.

\begin{dfn} \label{hs1} For a given set $S$ and family $\mathop{F} =\{S_1,...,S_k\}$, $S_i \subseteq S$, $\bigcup \limits_{i = 1}^k S_i=S$,
a hitting set $H \subseteq S$ of family $F$ is a set which has a non-empty intersection with each $S_i$,
i.e. $(\forall i \in \{1,...,k\}) H \bigcap S_i \ne \emptyset$.
\end{dfn}

The minimal hitting set problem is to find a hitting set of the minimal cardinality. This problem is equivalent to the set covering problem and it is known to be NP-hard \cite{np}.

\begin{dfn} \mbox{\rm(\cite{mil21,mil21a})} \label{mlb6} Value $mhs_<(G)$ is defined to be the minimal cardinality of hitting sets of family $\{W_{uv}, W_{vu} | uv \in E(G) \}$.\end{dfn}

From Lemma \ref{mlb5} and Definition \ref{mlb6} the next Corollary follows:

\begin{cor} \label{mlb7} \mbox{\rm(\cite{mil21,mil21a})} For every connected graph $G$ it holds $\beta_M(G) \ge mhs_<(G)$.
\end{cor}

Let $n$ and $k$ be positive integers ($n > k$) and $[n] = \{1,2,...,n\}$.
Then $k$-subsets are subsets of $[n]$ which have cardinality equal to $k$.
The Johnson graph $J_{n,k}$ is an undirected graph defined on all $k$-subsets of set $[n]$ as vertices,
where two $k$-subsets are adjacent if their intersection has cardinality equal to $k-1$.
Mathematically, $V(J_{n,k}) = \{ A | A \subset [n], |A|=k\}$ and
$E(J_{n,k}) = \{ AB | A,B \subset [n], |A|=|B|=k, |A \cap B|=k-1\}$.

It is easy to see that $J_{n,k}$ and $J_{n,n-k}$ are isomorphic, so
we shall only consider Johnson graphs with $n \ge 2k$.
The distance between two vertices $A$ and $B$ in $J_{n,k}$
can be computed by:

\begin{equation} \label{distj}
d(A,B) = |A \setminus B| = |B \setminus A| = k-|A \cap B|
\end{equation}

Considering (\ref{distj}), it is easy to see that Johnson graph $J_{n,k}$ is a $k (n-k)$ regular graph
of diameter $k$.
The following two propositions give the exact value of $\beta(J_{n,2})$ for $n \ge 6$ and an upper bound for
$\beta(J_{n,k})$ for $k \ge 3$.

\begin{pro} \label{metdkj2} \mbox{\rm(\cite{metdj1})} For $n \ge 6$, it holds $\beta(J_{n,2}) = \lceil \frac{2n}{3} \rceil$.
 \end{pro}

\begin{pro} \label{metdkj} \mbox{\rm(\cite{metdj})} For $n \ge 2k$ and $k \ge 3$, it holds $\beta(J_{n,k}) \le \lfloor \frac{k \cdot (n+1)}{k+1} \rfloor$.
 \end{pro}

The following three propositions give the exact value of $eqdim(J_{n,2})$, $eqdim(J_{2k,k})$ for odd $k$ and an upper bound for
$\beta(J_{n,3})$ for $n \ge 9$.

\begin{pro} \mbox{\rm(\cite{eqdimj})} $eqdim(J_{n,2}) = \begin{cases}
2, & n=4 \\
3, & n \ge 5 \end{cases}$
 \end{pro}

\begin{pro} \mbox{\rm(\cite{eqdimj})} For any odd $k \ge 3$ it holds $eqdim(J_{2k,k}) = \frac{1}{2} \cdot \binom{2k}{k}$.
 \end{pro}

\begin{pro} \mbox{\rm(\cite{eqdimj})} For $n \ge 9$ it holds  $eqdim(J_{n,3}) \le n-2$.
 \end{pro}

\section{New results}

This section is devoted to edge and mixed metric dimensions of Johnson graphs.
First, we construct a class of sets which are not edge resolving sets of $J_{n,k}$.
Based on this construction, we directly derive a necessary condition for a given set to be
an edge resolving set of $J_{n,k}$. Then, we obtain the exact values of edge and mixed metric dimensions of $J_{n,2}$.

\subsection{Some properties of edge resolving sets of $J_{n,k}$}

The next definition and lemma describe sets with relatively large cardinality, which, as well as all of their subsets, are not edge resolving sets of $J_{n,k}$.

\begin{dfn} \label{sprim} Let $x,y,z$ be some mutually distinct elements from $[n]$. Set $S_{x,y,z}' \subset V(J_{n,k})$
is defined as follows:
$$ S_{x,y,z}' = V(J_{n,k}) \setminus \{\{x,z\} \cup T, \{y,z\} \cup T\;| \; T \subset [n] \setminus \{x,y,z\},|T|=k-2\} $$
\end{dfn}

It should be noted that set $T$ goes through all subsets of set $[n] \setminus \{x,y,z\}$ of cardinality $k-2$, i.e.
$|S_{x,y,z}'| = \binom{n}{k} - 2 \cdot \binom{n-3}{k-2}$ for all $k \ge 2$.

\begin{lem} \label{noters} For arbitrary $k \ge 2$ and arbitrary mutually distinct $x,y,z \in [n]$, set $S_{x,y,z}'$ is not an edge resolving set of $J_{n,k}$.
\end{lem}
\begin{proof}Let $e_1$ and $e_2$ be edges from $E(J_{n,k})$ defined as follows.
Edge $e_1$ connects vertices $A = \{x,y\} \cup T^*$ and $B = \{x,z\} \cup T^*$, while
edge $e_2$ connects vertices $A = \{x,y\} \cup T^*$ and $C = \{y,z\} \cup T^*$,
where $T^* \subset [n] \setminus \{x,y,z\},|T^*|=k-2$.

We will prove that for each $X \in S_{x,y,z}'$ it holds $d(e_1,X) = d(e_2,X)$.
There are three possible cases for $X \in S'_{x,y,z}$.

\textit{\underline{Case 1}:} Let $X=\{x,y\} \cup T'$, where $T' \subset [n] \setminus \{x,y\}$ and $|T'|=k-2$. \\
Let $i = |T^* \cap T'|$, where $0 \le i \le k-2$. Then
$X \cap A = X \cap (\{x,y\}\cup T^*) = \{x,y\} \cup (T^* \cap T')$ and, according to (1), it holds
$d(X,A)= k - |X \cap A| = k-i-2$.
In the same way $d(X,B) = d(X,C)= k-i-1 - |T' \cap \{z\}|$ implying
$d(e_1,X) = d(e_2,X) = k-i-2$.
It shoud be noted that for $k=2$ it holds that $T^* = T' = \emptyset$
implying $i=0$, so $d(e_1,X) = d(e_2,X) = 0$.

\textit{\underline{Case 2}:} Let $X=\{p\} \cup T'$, where $p \in \{x,y,z\}$, $T' \subset [n] \setminus \{x,y,z\}$ and $|T'|=k-1$ \\
Let $i = |T^* \cap T'|$, where $0 \le i \le k-2$.
If $p=x$ then $X \cap A = X \cap (\{x,y\}\cup T^*) = \{x\} \cup (T^* \cap T')$ and, according to (1), it holds
$d(X,A)= k - |X \cap A| = k-i-1$. In the same way, $d(X,B)= k-i-1$
and $d(X,C)= k-i$ implying
$d(e_1,X) = d(e_2,X) = k-i-1$.
In a similar way, in subcases $p=y$ or $p=z$ the same result can be derived: $d(e_1,X) = d(e_2,X) = k-i-1$.

\textit{\underline{Case 3}:} Let $X=T'$, where $T' \subset [n] \setminus \{x,y,z\}$ and $|T'|=k$ \\
Let $i = |T^* \cap T'|$, where $0 \le i \le k-2$.
Then $X \cap A = X \cap T^* = T^* \cap T'$ and, according to (1), it holds
$d(X,A)= k - |X \cap A| = k-i$. In the same way, $d(X,B)= d(X,C)= k-i$ implying $d(e_1,X) = d(e_2,X) = k-i$.

Since, edges $e_1$ and $e_2$ have the same metric coordinates with respect to $S_{x,y,z}'$,
it follows that $S_{x,y,z}'$ is not an edge resolving set.
\end{proof}

If $\overline{S_{x,y,z}'}$ denotes the complement of $S_{x,y,z}'$, then we have the following corollary of Lemma \ref{noters}.

\begin{cor} \label{ers} If $S$ is an edge resolving set of $J_{n,k}$, then for every $x,y,z \in [n]$ it holds $S \cap \overline{S_{x,y,z}'} \ne \emptyset$.
\end{cor}

This corollary represents a necessary condition for set $S$ to be an edge resolving set of $J_{n,k}$.
Moreover, according to Corollary 1, each edge resolving set set of graph $J_{n,k}$ represents a
hitting set of family $F=\{\overline{S_{x,y,z}'} | x,y,z \in [n], x \ne y \ne z \ne x\}$.

Therefore, if $MHSP(F_{J_{n,k}})$ is the minimal cardinality
of hitting sets of family $F$ for graph $J_{n,k}$ the next Corollary holds:

\begin{cor} \label{elb3}
$\beta_E(J_{n,k}) \ge MHSP(F_{J_{n,k}})$.
\end{cor}

\subsection{Edge and mixed metric dimensions of $J_{n,2}$}

The next definition establishes a connection between arbitrary item $a \in V(G) \cup E(G)$
and set $S \subseteq V(G)$.

\begin{dfn} For any graph $G$ and any set $S \subseteq V(G)$, and item $a \in V(G) \bigcup E(G)$ let
$ZE(a,S) = \{ u \in S | d(a,u) = 0 \}$ and $ze(a,S)= | ZE(a,S)|$.
\end{dfn}

Obviously, if $ze(a,S) \ne ze(b,S)$ then items $a$ and $b$ have different metric coordinates with respect to $S$.
Additionally, for any vertex $u$ it holds $ze(u,S) \in \{0,1\}$ and
for any edge $e$ it holds $ze(e,S) \in \{0,1,2\}$.

The next theorem derives exact values of $\beta_E(J_{n,2})$ and $\beta_M(J_{n,2})$ for $n \ge 5$.

\begin{thm} \label{k2} For $n \ge 5$ it holds $\beta_E(J_{n,2}) =  \beta_M(J_{n,2}) = \binom{n}{2} - \lfloor \frac{n}{2} \rfloor$.
\end{thm}
\begin{proof}
Let $m = \lfloor \frac{n}{2} \rfloor$. \\
\textbf{\underline{Step 1}:}  $\beta_E(J_{n,2}) \ge \binom{n}{2} - m$ \\
Suppose the opposite, that $\beta_E(J_{n,2}) < \binom{n}{2} - m$.
Then there exists edge resolving set $S$, for which it holds $|S| < \binom{n}{2} - m$.
Since $|V(J_{n,2})| = \binom{n}{2}$, then $|V(J_{n,2}) \setminus S| > m$.
As both sides of this inequality are integer numbers it follows that $|V(J_{n,2}) \setminus S| \ge m+1$.
Since members of $V(J_{n,2}) \setminus S$ are pairs from set $\{1,2,...,n\}$ and their number is at least $m+1$,
then at least one element from $\{1,2,...,n\}$ has to repeat in at least two pairs from $V(J_{n,2}) \setminus S$,
i.e. $(\exists x)(\exists y)(\exists z) \{x,y\},\{x,z\} \in V(J_{n,2}) \setminus S$.
This fact implies

$$S \subseteq V(J_{n,2}) \setminus \{ \{x,y\},\{x,z\} \} = S_{x,y,z}'$$

By Lemma \ref{noters}, $S_{x,y,z}'$ is not an edge resolving set of $J_{n,2}$,
and consequently the same holds for its subset $S$, which is in contradiction with
the starting assumption. Therefore,
$\beta_E(J_{n,2}) \ge \binom{n}{2} - \lfloor \frac{n}{2} \rfloor$. \\

\textbf{\underline{Step 2}:}  $\beta_M(J_{n,2}) \le \binom{n}{2} - m$\\
Let us consider set $S = V(J_{n,2}) \setminus \{ \{2i-1,2i\} | 1 \le i \le m\}$
for which $|S| = \binom{n}{2} - m$.
We will prove that $S$ is a mixed resolving set for $J_{n,2}$, i.e. that any item
$a$ from $V(J_{n,2}) \bigcup E(J_{n,2})$ has unique metric coordinates with respect to $S$.
There are five cases for item $a$ depending on $a \in V(J_{n,2})$ or $a \in E(J_{n,2})$ and on value
$ze(a,S)$.

\textit{\underline{Case 1}:}  $a \in E(J_{n,2})$ and $ze(a,S) = 0$\\
Let $a = XY \in E(J_{n,2})$. As $ze(a,S) = 0$ it follows that $X \notin S$, so
$X$ can be represented as $X = \{2i-1,2i\}$ for some $i$ such that $1 \le i \le m$.
Also it follows that $Y \notin S$, so
$Y$ can be represented as $Y = \{2j-1,2j\}$ for some $j$ such that $j \ne i$ and $1 \le j \le m$.
Since $X \bigcap Y = \{2i-1,2i\} \bigcap \{2j-1,2j\} = \emptyset $
then $XY \notin E(J_{n,2})$ which is in contradiction with $a = XY \in E(J_{n,2})$.
Therefore, for any $a \in E(J_{n,2})$ it follows that $ze(a,S) > 0$.
Hence, Case 1 actually does not exist.

\textit{\underline{Case 2}:}  $a \in E(J_{n,2})$ and $ze(a,S) = 1$\\
Let $a = XY \in E(J_{n,2})$. As $ze(a,S) = 1$ without loss of generality we can assume that $X \in S$
and $Y \notin S$. Since $Y \notin S$ then $Y$ can be represented as $Y = \{2j-1,2j\}$
for some $j$ such that $1 \le j \le m$.
Let $X=\{q,r\}$. As $XY$ is an edge, it holds that $2j-1 \in X$ or $2j \in X$ (not both)
so $q+r \ne 4j-1$.
Without loss of generality we can assume that $q \in \{2j-1,2j\}$.
Let $\overline{q} = 4j-1-q$. Then $\overline{q} \ne q$
and $\overline{q} \in \{2j-1,2j\}$.
As $q+r \ne 4j-1$ and $q+\overline{q} = 4j-1$ it follows that
$\overline{q} \ne r$.
Therefore, $\overline{q} \notin X$.

Since $d(a,X)=0$, we should only compare metric coordinates of $XY$ with metric coordinates of any item
$b$ such that $ZE(b,S) = \{X\}$. This is satisfied when:
\begin{itemize}
\item vertex $b=X$, or
\item edge $b=XZ$, for any vertex $Z \notin S$.
\end{itemize}

Since $Z \notin S$, then $Z$ can be represented as $Z = \{2l-1,2l\}$, $l \ne j$.
It is clear that $YZ \notin E(J_{n,2})$ since $Y \bigcap Z = \emptyset$.

Since $XZ \in E(J_{n,2})$ it follows that $r \in \{2l-1,2l\}$.
Let  $\overline{r}=4l-1-r$ then $\overline{r} \in \{2l-1,2l\}$.
As $q \in \{2j-1,2j\}$ and $j \ne l$ then $\overline{r} \ne q$.
Since $\overline{r} \ne r$ and $\overline{r} \ne q$ then $\overline{r} \notin X$.

Since $n \ge 5$ we can choose arbitrary index $p$ which is not in set
$\{2l-1,2l,2j-1,2j\}$. Let
$T_1=\{p,\overline{q}\}$ and $T_2=\{p,\overline{r}\}$.

From $\overline{q} \in \{2j-1,2j\}$, $\overline{r} \in \{2l-1,2l\}$
and $p \notin \{2l-1,2l,2j-1,2j\}$ it follows that
$T_1, T_2 \in S$.

Let us compare metric coordinates of edge $a=XY$ and vertex $b=X$
with respect to $T_1$:

$d(X,T_1) = d(\{q,r\},\{p,\overline{q}\}) = 2$
and $d(XY,T_1) = 1$, since \\
$d(Y,T_1)=d(\{2j-1,2j\},\{p,\overline{q}\})=1$ because $\overline{q} \in \{2j-1,2j\}$.

Let us compare metric coordinates of edge $a=XY$ and vertex $b=XZ$
with respect to $T_2$:

$d(X,T_2) = d(\{q,r\}, \{p,\overline{r}\}) = 2$ since $p, \overline{r} \notin X$. \\
$d(Y,T_2) = d(\{2j-1,2j\}, \{p,\overline{r}\}) = 2$ since $p, \overline{r} \notin \{2j-1,2j\}$. \\
$d(Z,T_2) = d(\{2l-1,2l\}, \{p,\overline{r}\}) = 1$ since $\overline{r} \in \{2l-1,2l\}$. \\
Therefore, $2 = d(XY,T_2) \ne d(XZ,T_2) = 1$.

Hence, edge $XY$ has different metric coordinates with respect to $S$
compared to edge $XZ$ and vertex $X$,
so $XY$ has unique metric coordinates with respect to $S$.

\textit{\underline{Case 3}:}  $a \in E(J_{n,2})$ and $ze(a,S) = 2$\\
Let $a = XY$ and $ZE(a,S)=\{X,Y\}$. Edge $a=XY$ has unique metric coordinates with
respect to $S$ compared to any vertex of $v \in V(J_{n,2})$,
since $ze(v,S) \le 1$. Also $a=XY$ has unique metric coordinates with
respect to $S$ compared to any other egde $e \in E(J_{n,2})$ since
$ZE(e,S) \ne \{X,Y\}$.

\textit{\underline{Case 4}:}  $a \in V(J_{n,2})$ and $ze(a,S) = 0$\\
Let $a = X \in V(J_{n,2})$. As $ze(a,S) = 0$ it follows that $X \notin S$ so
$X$ can be represented as $X = \{2i-1,2i\}$ for some $i$ such that $1 \le i \le m$.
We should compare metric coordinates of $X$ with metric coordinates of any other
item $b$ such that $ze(b,S) = 0$. According to Case 1 for any edge $e \in E(J_{n,2})$ it holds $ze(e,S) > 0$.
Therefore it is sufficient to compare $X$ only with vertex $b=Z$ such that $ze(b,S) = 0$.
As $Z \notin S$ then $Z=\{2j-1,2j\}$, where $j \ne i$ and $1 \le j \le m$. Let $p \notin \{2i-1,2i,2j-1,2j\}$
and $T=\{p,2i\}$. Since $p \ne 2i-1$ then $T \in S$. Since $d(X,T) = d(\{2i-1,2i\},\{p,2i\}) = 1$
and $d(Z,T)= d(\{2j-1,2j\},\{p,2i\}) = 2$ then $X$ and $Z$ have different metric coordinates with respect to $S$,
so $X$ has unique metric coordinates with respect to $S$.

\textit{\underline{Case 5}:}  $a \in V(J_{n,2})$ and $ze(a,S) = 1$\\
Let $a = X \in V(J_{n,2})$. As $ze(a,S) = 1$ it follows that $X \in S$.
Since $ZE(X,S)=\{X\}$ then vertex $X$ has unique metric coordinates with
respect to $S$ compared to any other vertex of $V(J_{n,2})$.

Therefore, it is sufficient to compare metric coordinates of $X$ with metric coordinates of any edge
$e$ such that $ZE(e,S) = \{X\}$. Let $e=XY$ such that $Y \notin S$,
so $Y = \{2j-1,2j\}$, and let $X=\{q,r\}$.
It holds $2j-1 \in X$ or $2j \in X$ (not both). Let us assume that $q \in \{2j-1,2j\}$.
Let $\overline{q} = 4j-1-q$. Then $\overline{q} \in \{2j-1,2j\}$ and $\overline{q} \notin X$.

Let $p \notin \{r,2j-1,2j\}$ and $T=\{p,\overline{q}\}$. Since $p \ne 2j-1$ and $p \ne 2j$ then $T \in S$.
Therefore $d(X,T) = d(\{q,r\},\{p,\overline{q}\}) = 2$
and $d(XY,T) = 1$ since $d(Y,T)=d(\{2j-1,2j\},\{p,\overline{q}\})=1$.
Hence, $X$ and $XY$ have different metric coordinates with respect to $S$,
so $X$ has unique metric coordinates with respect to $S$.

Since $\beta_E(G) \le \beta_M(G)$ holds for every connected graph $G$, then the proof is completed.
\end{proof}

Since Theorem \ref{k2} holds for $n \ge 5$, the only remaining case is $n=4, k=2$.
The exact values of $\beta_E(J_{4,2})$ and  $\beta_M(J_{4,2})$ are given in the next observation.

\begin{obs} \label{k2n4} It is found, by a total enumeration, that $\beta_E(J_{4,2}) = \beta_M(J_{4,2})$ = 5,
with one of mixed metric bases $S = \{ \{1,2\},  \{1,3\},  \{1,4\},  \{2,3\},  \{2,4\} \}$.
\end{obs}

From Proposition \ref{metdkj2} and Theorem \ref{k2}, it follows that edge and mixed metric dimensions of Johnson graphs are substantially larger than the metric dimension of $J_{n,2}$.

It is interesting to examine values $\beta_E(J_{n,k})$ and  $\beta_M(J_{n,k})$ for $k \ge 3$,
which are presented in Table 1. The first four columns contain $n$, $k$, $|V(J_{n,k})|$ and $|E(J_{n,k})|$.
The next three columns contain exact values of $\beta(J_{n,k})$, $\beta_E(J_{n,k})$ and $\beta_M(J_{n,k})$.
Theoretical lower bound {\em Ed} for $\beta_E(J_{n,k})$ obtained by Theorem \ref{elb2} is presented in column eigth.
The next two columns, denoted by {\em Mi1} and {\em Mi2}, contain theoretical lower bounds for $\beta_M(J_{n,k})$ obtained by Corollary \ref{mlb3}
and Theorem \ref{mlb4}, respectively. The next two columns are obtained as relaxations of ILP formulation
of edge and mixed metric dimension problems, respectively. Last two columns, $MHSP$ and $mhs_<$ present the hitting set lower bounds of
$\beta_E(J_{n,k})$ and $\beta_M(J_{n,k})$ obtained by Corollary \ref{elb3} and Corollary \ref{mlb7},
respectively.

All computational results presented in columns 5-7 and columns 11-14 of Table 1 are obtained using the well-known integer programming formulations
and their continous relaxations \cite{hitilp,metdilp}. The computations are performed by Gurobi 11.0.1 solver
\cite{gur} on computer AMD Ryzen 5, model 5600H with 6 cores (12 logical processors) at 3.4 GHz and 8 GB of RAM memory.
The running times (in seconds) are presented in Table 2.

\begin{table}
\caption{Exact values and bounds of metric dimensions for $J_{n,k}$, $k \ge 3$} \label{johnk3}
 \small
\begin{center}
\begin{tabular}{|c|c|c|c|c|c|c|c|c|c|c|c|c|c|}
\hline
$n$ & $k$ & $|V|$ & $|E|$ & \multicolumn{3}{|c|}{Exact value} & \multicolumn{3}{|c|}{Th. LB} & \multicolumn{2}{|c|}{Rel. LB}  & \multicolumn{2}{|c|}{Hitting LB} \\
\cline{5-14}
& & & & $\beta $ & $\beta_E$ & $\beta_M$ & Ed & Mi1 & Mi2 & Ed & Mi & {\tiny $MHSP$} & $mhs_<$ \\
 \hline
6  & 3  & 20  & 90   & 4 & 8   & 8  & 5 & 5  & 4 & 4 & 4 & 6   & 4  \\
 \hline
7  & 3  & 35  & 210  & 5 & 10  & 10 & 5 & 5  & 5 & 5 & 5 & 10  & 5  \\
 \hline
8  & 3  & 56  & 420  & 5 & 14  & 14 & 5 & 5  & 6 & 6 & 6 & 12  & 6  \\
8  & 4  & 70  & 560  & 5 & 9   & 9  & 5 & 6  & 5 & 4 & 4 & 6   & 5  \\
 \hline
9  & 3  & 84  & 756  & 6 & 18  & 18 & 6 & 6  & 6 & 7 & 7 & 16  & 6  \\
9  & 4  & 126 & 1260 & 6 & 11  & 11 & 6 & 6  & 6 & 5 & 5 & 9   & 6  \\
 \hline
\end{tabular}
\end{center}
\end{table}

\begin{table}
\caption{Gurobi running times in seconds} \label{johnk3r}
 \small
\begin{center}
\begin{tabular}{|c|c|r|r|r|r|r|r|r|}
\hline
$n$ & $k$ &  \multicolumn{3}{|c|}{Exact value} & \multicolumn{2}{|c|}{Rel. LB}  & \multicolumn{2}{|c|}{Hitting LB} \\
\cline{3-9}
& &  $\beta $ & $\beta_E$ & $\beta_M$ & Ed & Mi & {\tiny $MHSP$} & $mhs_<$ \\
 \hline
6  & 3  & 0.02 & 0.13 & 0.06 & 0.03  & 0.03 & 0.01 & 0.02 \\
 \hline
7  & 3  & 0.13 & 1.46 & 0.45 & 0.16  & 0.28 & 0.06 & 0.02 \\
 \hline
8  & 3  & 0.32 & 27.6 & 19.3 & 1.18  & 1.93 & 0.16 & 0.05 \\
8  & 4  & 0.11 & 66.2 & 56.3 & 3.23  & 4.20 & 0.08 & 0.03 \\
 \hline
9  & 3  & 0.89 & 3224 &  2534 & 7.77  & 9.88 & 1.80 & 0.05 \\
9  & 4  & 1.94 & 13554 & 17308 & 37.8 & 47.5 & 0.64 & 0.09 \\
 \hline
\end{tabular}
\end{center}
\end{table}

Values in Table 1 lead us to following observations. Theoretical lower bounds
are general and can be easily calculated, but
the obtained bounds are far from exact values. The lower bounds based
on LP relaxations and $mhs_<$ are similar to theoretical bounds.
From Table 1, it is evident that the best lower bound is $MHSP(F_{J_{n,k}})$,
which is tight, as $\beta_E(J_{7,3}) = \beta_M(J_{7,3}) = 10 = MHSP(F_{J_{7,3}})$.

From Table 2, it can be seen that the running times for obtaining
exact values $\beta_E(J_{n,k})$ and $\beta_M(J_{n,k})$ are
much larger than the running time for $\beta(J_{n,k})$.
The running times for computing $MHSP(F_{J_{n,k}})$ lower bounds are
several orders of magnitude lower than the running times for
computing the exact values $\beta_E(J_{n,k})$. Although
the problem of finding $MHSP(F_{J_{n,k}})$ is also NP-hard,
Table 1 and Table 2 indicate that $MHSP(F_{J_{n,k}})$ is a good compromise
between bound's quality and the running time.

\begin{obs} \label{n11k5} In the case of $J_{11,5}$, it is calculated by Gurobi 11.0.1 solver that
$MHSP(F_{J_{11,5}})=10$ for 34.19 seconds. Gurobi 11.0.1 solver could not obtain
the exact value of $\beta_M(J_{11,5})$ due to "out of memory" status. An upper bound for $\beta_M(J_{11,5})$
equal to 10 is obtained by a VNS metaheuristic.
Since $MHSP(F_{J_{11,5}})=10$ and $MHSP(F_{J_{n,k}}) \le \beta_E(J_{n,k}) \le \beta_M(J_{n,k})$
it follows that $\beta_E(J_{11,5})=\beta_E(J_{11,5})=10$.
\end{obs}

Theorem 1 and Table 1 lead us to the following Hypothesis.

\begin{hyp} \label{hy1}For $k \ge 3$ it holds $\beta_E(J_{n,k}) = \beta_M(J_{n,k})$.
\end{hyp}

\section{Conclusions}

In this paper, both edge and mixed metric dimensions of Johnson graphs $J_{n,k}$ are considered.
Exact value for $\beta_E(J_{n,2})$ and $\beta_M(J_{n,2})$ equal to $\binom{n}{2} - \lfloor \frac{n}{2} \rfloor$
are found for all $n \ge 5$, while for $n=4,k=2$ both values are equal to 5.

Further work can be directed to finding edge and mixed metric dimensions of Kneser graphs, as well as, of some other interesting classes of graphs.

\bibliographystyle{elsarticle-num}

\section*{Acknoledgements}
\end{document}